\documentclass[]{article}

\usepackage{xcolor}
\usepackage{amsthm,amsmath,amssymb}
\usepackage{hyperref}
\usepackage[capitalise, nameinlink]{cleveref}
\usepackage{enumerate}
\usepackage{graphicx}
\def\mpfile#1#2{\includegraphics{#1-#2-mps.pdf}}

\theoremstyle{plain}
\newtheorem{theorem}{Theorem}
\newtheorem{claim}[theorem]{Claim}
\newtheorem{proposition}[theorem]{Proposition}
\newtheorem{cor}[theorem]{Corollary}
\newtheorem{lemma}[theorem]{Lemma}
\newtheorem{observation}[theorem]{Observation}

\theoremstyle{definition}
\newtheorem{definition}{Definition}
\newtheorem{conjecture}{Conjecture}

\newtheorem{remark}{Remark}

\title{Growing Trees and Amoebas' Replications}

\author{
Vladimir Gurvich\thanks{National Research University Higher School of Economics;
RUTCOR, Rutgers the State University of NJ, USA;  \texttt{vladimir.gurvich@gmail.com}.}\and
Matjaž Krnc\thanks{FAMNIT and IAM, University of Primorska; texttt{matjaz.krnc@upr.si}.}\and
Mikhail Vyalyi\thanks{Federal Research Center “Computer Science and Control” of the Russian Academy of Science; National Research University Higher School of Economics; Moscow Institute of Physics and Technology; \texttt{vyalyi@gmail.com}.}}

\newcommand{\cA}{\mathcal{C}}

\def\orb#1{\mathop{\mathrm{Orb}}(#1)}
\def\diam#1{\mathop{\mathrm{diam}}(#1)}
\DeclareMathOperator{\supp}{supp}

\newcommand{\tm}{\widetilde m}

\begin{document}
\maketitle
\begin{abstract}
{An amoeba is a tree together with instructions how to iteratively grow trees by adding paths of a fixed length $\ell$.
This paper analyses such  a growth process.
An amoeba is mortal if all versions of the process are finite,
and it is immortal if they are all infinite.

We obtain some necessary and some sufficient conditions for mortality.
In particular, for growing caterpillars in the case $\ell=1$ we characterize mortal amoebas.
We discuss variations of the mortality concept, conjecture that  some of them are equivalent, and support this conjecture for $\ell\in\{1,2\}$.
}
\end{abstract}



\maketitle

\section{Introduction}

A classical result of Dirac from 1961
\cite{MR130190}
states that every (non-null) chordal graph has a vertex whose neighborhood is a clique. 
This result was generalized in the literature in various ways \cite{MR485552,MR2659379,MR2057267,MR1927566,MR1626534,MR408312}. 
In 2019, Beisegel et al.~\cite{BCGMS19} proposed a common generalization of all these results by introducing the concept of \emph{avoidable paths}. 
An induced path $P$ in a graph $G$ is said to be \emph{avoidable} if every extension of $P$ is contained in an induced cycle.
Here, an \emph{extension} of $P$  is any induced path in $G$ that can be obtained by extending $P$ by one edge from each endpoint.
Beisegel et al.\ conjectured that for every positive integer $k$, every graph that contains an induced $k$-vertex path also contains an avoidable induced $k$-vertex path, and proved the statement for the case $k=2$.
The general conjecture was proved in 2020 by Bonamy et al.~\cite{MR4245221}.
A further strengthening was given by Gurvich et al.~\cite{GKMV20+}, who showed that in every graph, every induced path can be transformed into an avoidable one via a sequence of shifts (where two induced $k$-vertex paths are said to be \emph{shifts} of each other if their union is an induced path with $k+1$ vertices).

In \cite{gurvich2023avoidability} the concept of avoidability was applied to general graphs rather than only paths.
A graph $H$ with roots $s,t\in V(H)$ is said to be \emph{inherent} if every graph $G$ that contains an induced copy of $H$, also contains an avoidable induced copy of $H$.
In the special case when $H=(v_1,\dots,v_n)$ is a path with $s=v_1$ and $t=v_n$, the result of Bonamy et al. \cite{MR4245221} states that $(H,s,t)$ is inherent. 
In \cite{gurvich2023avoidability} the authors develop a so-called pendant extension  method for showing that a fixed $2$-rooted graph is not inherent, and identify many such $2$-rooted graphs.
This method is based on the process of growing trees.\footnote{The case of trees is the only non-trivial case: using the pendant extension method it is easily shown that the process stops for all connected graphs which contain a cycle (see also  \cref{rem2}).} 
If this process halts, a particular $2$-rooted tree is non-inherent. 
It is open if the halting problem for this process is decidable.

The above process is of independent interest. 
In this paper we generalize it in two ways. First, we allow  arbitrarily many roots. Second, a step  of our process may add an arbitrary path to a root, rather than a single edge.

An amoeba is a tree with a number  of roots assigned to each vertex. 
A tree $T$ may have several copies of a particular amoeba.
Each one requires a `certain growth space' in  $T$. If it cannot grow within $T$, the amoeba will live on by expanding $T$ --  attaching new paths at the points of its roots.
In this way new copies of the amoeba may appear in the expanded $T$. 
Alternatively, if every copy of  the amoeba in $T$ has enough growth space, the $T$ will not grow anymore. 
We say that $T$ confines such an amoeba.
To specify the finiteness of growth of a given amoeba we introduce several notions of \emph{mortality}. 
In this paper we present some necessary and some sufficient conditions for amoeba mortality.

Our process of  growing trees has relations to  several known models:

\medskip
\noindent
\textbf{String and graph generation models}$\quad$ are intensively studied in theoretical computer science, formal language theory, and mathematical biology. 
There are many models of this sort: cellular automata, rewriting systems, graph grammars, etc. \cite{citeulike:106131, DBLP:books/el/leeuwen90/DershowitzJ90, LINDENMAYER1968280, DBLP:conf/gg/1997handbook}. In particular, 
our model resembles  graph relabelling systems with forbidden contexts \cite{DBLP:conf/gg/1997handbook}. Most graph generation models use labelling of graphs. In contrast, we propose a simple rule of tree growth governed by the structure of subtrees and does not use labels on vertices and/or edges. The resulting process is interesting from the graph-theoretical point of view and possibly might impact  the area of graph generation models.
\medskip

\noindent
\textbf{The Galton-Watson process}$\quad$ is a probabilistic process of growing trees.
It is a foundational concept in the realm of probability theory, specifically in the study of family trees and the dynamics of progeny within populations \cite{gw1}. Introduced by Francis Galton and Henry Watson \cite{gw2,gw1}, this stochastic process serves as a mathematical model to examine the transmission of traits, characteristics, or lineage through successive generations of families \cite{gw1,mitov_yanev_1984}. The key idea revolves around the probability distribution of offspring within a family, where individuals independently give rise to a random number of descendants \cite{gw1}. This model has proven particularly valuable in fields such as genetics and demography \cite{gw2}, providing insights into the likelihood of lineage continuation or extinction \cite{gw1}. The Galton-Watson process remains an essential tool for understanding the unpredictable nature of family structures and the variability in the number of progeny across generations \cite{mitov_yanev_1984}.

In contrast to the Galton-Watson process, the replication criteria of the amoeba replication process is not dictated by a  probability distribution. Instead, the expansion is orchestrated by the existence of a specific subgraph within the existing tree, acting as a trigger for the generation of additional leaves or progeny. 
This departure from stochastic randomness to a combinatorial criterion opens up new avenues for exploring the interplay between structure and growth of trees.

\vspace{0.4cm}
This paper is structured as follows. In the beginning of \cref{sec:amoebas} we introduce the main concepts of the paper, such as \emph{amoeba}, its \emph{extension}, \emph{growth}, \emph{mortality}, \emph{immortality}, and \emph{confinement}. 
In \cref{subs:init} we relate mortality and confinement conditions, while in \cref{subs:orbit} we discuss relations between the symmetries of the underlying graphs and amoeba mortality.
In \cref{sec:1-extensions} we focus on the mortality conditions for amoebas growing in the simplest way, namely, by attaching pendant edges.
We identify degree conditions which are necessary for any such amoeba to be not mortal in \cref{subs:degree}, while in \cref{subs:caterpillar} we describe the not mortal amoebas whose underlying graph is a caterpillar. 
We conclude the paper with \cref{sec:future}, in which  we propose some conjectures as well as directions for future work.

\section{Amoebas}
\label{sec:amoebas}

    An \emph{amoeba}\footnote{This name was recently given to a different class of graphs~\cite{DBLP:journals/combinatorics/CaroHM23}. We think that these two classes cannot be confused.} $A=(T,m)$ is defined as a pair consisting of a tree $T$, together with the 
    \emph{root function} 
    $m:V(T)\to\mathbb N$ mapping every vertex $v$ of $T$ to a non-negative integer $m(v)$ called the \emph{multiplicity} of $v$.

    A vertex from $A$ of positive multiplicity is called a root of $A$. 
    Furthermore, we define \emph{the number of roots} 
 of $A$ to be $\sum_{v\in V(H)} m(v)$. 
    The set of roots of $A$ is denoted by $\supp(m)$.

  An \emph{$\ell$-extension of amoeba} $(H,m)$ is a graph $G$ obtained from $H$ by applying to each vertex $v$ the following procedure: take $m(v)$ disjoint paths
    $P_{\ell+1}$
and  identify one endpoint of each with $v$. The paths are called \emph{extension paths}.

A \emph{copy of amoeba} $(H,m)$ in a graph $G$ is an amoeba $(H',m')$ where $H'\le G$, and $H\cong H'$ such that the functions $m,m'$ coincide under the corresponding isomorphism.
When there is no possibility of ambiguity, we denote an amoeba and its copy  by the same notation.

Let
  $A$
  be a copy of an amoeba in a tree $G$. 
  Then an \emph{$\ell$-growth} $G'$ of $A$ in $G$ is a tree 
containing both $G$ and an $\ell$-extension of $A$ and having the minimal possible number of edges.
Note that an {$\ell$-growth} is not uniquely defined.

If, for a copy $A$ of an amoeba in graph $G$,  all growths of $A$ in $G$ coincide with $G$, then the copy is called \emph{dead}. Otherwise, the copy is called \emph{alive}.

Note that a copy of an amoeba may have non-isomorphic growths; see an
example in Fig.~\ref{pic:growth}.

\begin{figure}[!h]
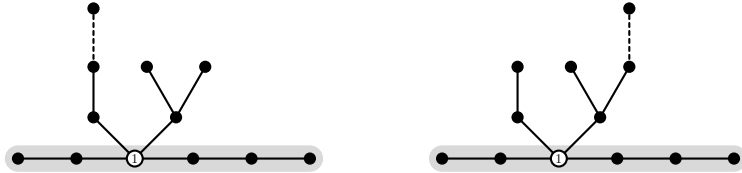

  \centering
    \mpfile{growth}{10}\qquad\qquad \mpfile{growth}{11}
    
    \caption{Non-isomorhic 3-growths of the same copy of an amoeba. The copy
    is shadowed in grey, the new edge is dashed}
    \label{pic:growth}
  \end{figure}

For an amoeba $A=(H,m)$, an $\ell$-\emph{growth sequence} of  $A$ starting in a tree $T$
 is a (possibly infinite)  sequence of trees $T_0,T_1,\dots$ such that 
  $T_0=T$ and $T_{i+1}$ is an $\ell$-growth of a copy of $A$ in $T_i$, and $T_{i+1}\neq T_i$. 
 An $\ell$-{growth sequence} is \emph{maximal} if it cannot be extended.

It will be useful to define the following  special case of maximal $\ell$-growth sequences:
for an amoeba $A=(H,m)$, an $\ell$-\emph{generation sequence} of  $A$  is a maximal $\ell$-growth sequence starting in $H$. See an example in Fig.~\ref{pic:generation}.

\begin{figure}[!h]
  \centering
 \includegraphics{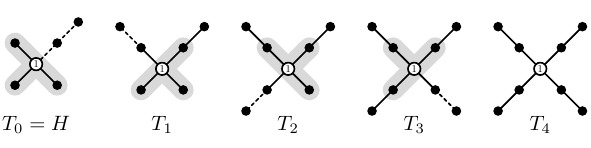}  
    \caption{A generation sequence for an amoeba $H$ with $\ell=2$. The $\ell$-growth of $T_4$ is $T_4$}
    \label{pic:generation}
\end{figure}

If $\ell$ is clear from the context it will be omitted in notation for extensions, growths, generation sequences, etc.

Let $\ell$ be the extension parameter, let $A$ be an amoeba, and let $T$ be a tree. We introduce the following definitions.
\begin{itemize}
    \item Amoeba $A$ is  \emph{immortal} w.r.t $T$ if all maximal $\ell$-growth sequences of $A$ starting in $T$ are infinite.
    \item Amoeba $A$ is \emph{mortal} w.r.t $T$ if all maximal $\ell$-growth sequences of $A$ starting in $T$ are finite. 
    \item When $A = (H,m)$ is (im)mortal w.r.t $H$, we say that $A$ is (im)mortal.
    \item  A tree $T$ is \emph{confining} for an amoeba $A $ if it admits a copy of $A$ and all copies of $A$ in $T$ are dead.
\end{itemize}
Note that any $\ell$-growth sequence  starting in a confining tree is of length $1$. {It remains open whether all generation sequences of $A$ can be infinite.}
\begin{itemize}
  \item An amoeba is  \emph{totally confined} if for each tree $T$ containing a copy of $A$ there exists a confining tree $T'$ containing a copy of $T$.
\end{itemize}

\begin{remark}\label{rem:confining-graphs}
    Confining trees can be replaced by confining
    graphs.
    Totally confined amoebas can be defined in this more general setting too. 
\end{remark}

\begin{remark}\label{rem2}
One can use general graphs rather than trees in the extension process (in \cite{gurvich2023avoidability} it was done for general two-rooted graphs). 
    Since attaching paths to a graph does not change its set of cycles, this process is finite for non-trees.
    As we mainly study mortality, we restrict our attention to the case of trees.
\end{remark}

\subsection{Initial observations}\label{subs:init}
Note that a copy of an amoeba may have non-isomorphic growths, which lead to essentially different growth sequences.
We are going to show that this effect can take place only for early stages of the growth process and only for amoebas that are small enough w.r.t the growth parameter $\ell$.
An example in Fig.~\ref{pic:growth} cannot be seen in generation sequences, as we will prove. It is caused by the minimality condition.

 Let $T'$ be an $\ell$-growth of an alive copy of an amoeba $A = (H,m)$ in a tree $T$, and let $A'$ be the corresponding extension of $A$ contained in $T'$, and $P$ be a path $v_0, v_1, \dots, v_{\ell}$  in the extension starting in a vertex $v_0\in V(H)$ such that $\{v_{\ell-1}, v_{\ell}\}\notin E(T)$. Take a connected component of the forest $T'-H$ containing $v_1$. It is a tree $T''$ rooted at $v_1$, which is called an \emph{area} of $P$. Due to the minimality condition, the depth of $T''$ (the maximal length of a path from the root to a leaf) is $\ell-1$.  Moreover, $v_1, \dots, v_{\ell}$ is a unique path of length $\ell-1$ from the root $v_1$. Note that other edges in $T'$ are the edges of the original tree $T$.  In this way we come to the following conclusion.

 \begin{proposition}
   The intersection of an area of an extension path and the original graph is a tree of depth at most $\ell-1$.
 \end{proposition}

\def\I{{\mathcal{I}}}

 Given a tree $T$ and its subtree $T'$, there is a partial order on $V(T)\setminus V(T')$ defined as follows: $u\leq^{T'} v$ if $u$ belongs to the path from $v$ to $T'$. For each $v \in V(T)\setminus V(T')$ the set $\{u: v\leq^{T'} u\}$ induces a tree $\I_{v}^{T'}$ rooted at $v$. 
 We will not mention $T'$ in the exponent when it is clear from the context.

 \begin{lemma}\label{lm:depth-bound}
   Let  $T_0,T_1,\dots, T_k$ be an $\ell$-growth sequence for an amoeba  $A=(H,m)$ and $v\in V(T_k)\setminus V(T_0)$ such that $\I_v^{T_0}$ has depth at most $\ell-1$. If for some vertex $u$ in $\I_v^{T_0}$ there are two different paths $P$ and $P'$ of equal lengths from $u$ to a leaf, then
   \[
\text{\textup{diam}} (H) \leq 2\ell-4\ \text{and}\ \text{\textup{diam}} (T_0 )\leq 2\ell-6.   
   \]
 \end{lemma}
 \begin{proof}
 W.l.o.g. we can assume  that $P$ and $P'$ have no common edges. 
 Indeed, otherwise we can change the vertex $u$ by the last common vertex of $P$ and $P'$.
 Choose the maximal $i$ such that at least one edge of $P\cup P'$ is not in $T_i$. 
 W.l.o.g. we can assume that $e=(u_0, u_1)$ is the last edge of $P$ and $e$ is not in $T_i$.
 It means that $e$ is added in a growth in $T_i$ of a copy  $A = (H,m)$ of the amoeba.
 So an extension of $A$ contains an extension path from a root $r$ of $A$ to $u_1$. 
 We claim  that $u$ is not  on the path from $r$ to $u_1$ that has no common edges with $P'$.
 Indeed, otherwise $P'$ does not intersect $H$ and the extension path can be changed to the path containing  $P'$,  contradicting the minimality condition; see Fig.~\ref{pic:depth}(a). 
 Also, $r$ is  not  on $P$. Indeed, otherwise the extension should add more edges, since the length of the whole path $P$ is less than $\ell$; see Fig.~\ref{pic:depth}(b).

\begin{figure}[!h]
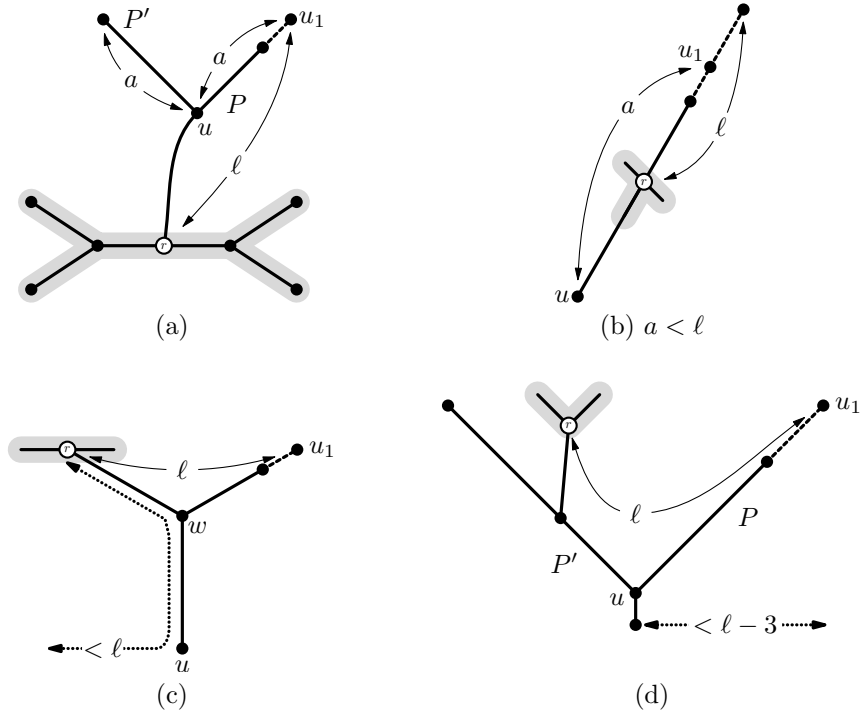

  \centering

  \begin{tabular}{c@{\qquad\qquad}c}
    \mpfile{growth}{20} &\mpfile{growth}{21}\\
    (a) & (b) $a<\ell$\\[5mm]
    \mpfile{growth}{22} &\mpfile{growth}{23} \\
    (c) Option \eqref{Option1} & (d) Option \eqref{Option2}
  \end{tabular}
   
    \caption{Four ways of extension}
    \label{pic:depth}
  \end{figure}

In contrast, the next two options are possible and mutually exclusive:
\begin{enumerate}[(i)]
\item\label{Option1}  vertices $u$, $u_1$, $r$ are end-points of paths starting in a vertex $w$; see Fig.~\ref{pic:depth}(c);
\item\label{Option2}  the path from $r$ to $u_1$ has a common edge with $P'$; see Fig.~\ref{pic:depth}(d).
\end{enumerate}

   In both cases $A$ is contained in the subtree rooted at $r$ since the path from $r$ to $u_1$ has no common edges with $A$. This subtree is inside $\I^{T_0}_u$ so its depth is at most $\ell-2$, which implies that diameter of $H$ is at most $2\ell-4$.
   Also, any  path starting in $r$ and containing $w$ and $u$ 
   should have length at most  $\ell-1$, by the minimality condition. 
   Since $T_i$ contains $T_0$ and the path from $r$ to $T_0$ has at least 3 edges, the diameter of $T_0$ is at most $2\ell-6$.
 \end{proof}

\cref{lm:depth-bound} puts restrictions on extensions using `new' edges only. Non-isomorphic $\ell$-growths are possible in generation sequences even if $H$
has the diameter at least $2\ell-6$; see an example in Fig.~\ref{pic:ambiguity}.

\begin{figure}[!h]
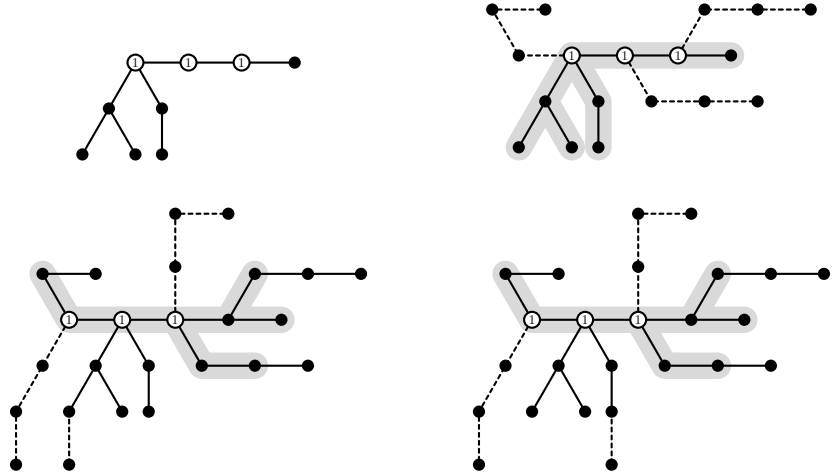

  \centering
  \begin{tabular}{c@{\qquad\qquad}c}
    \mpfile{amb}{10} &\mpfile{amb}{11}\\[5mm]
    \mpfile{amb}{12} &\mpfile{amb}{13}
  \end{tabular}
  \caption{An amoeba admitting ambiguous $3$-extensions in a generation sequence}
  \label{pic:ambiguity}
\end{figure}

  \begin{figure}[!h]
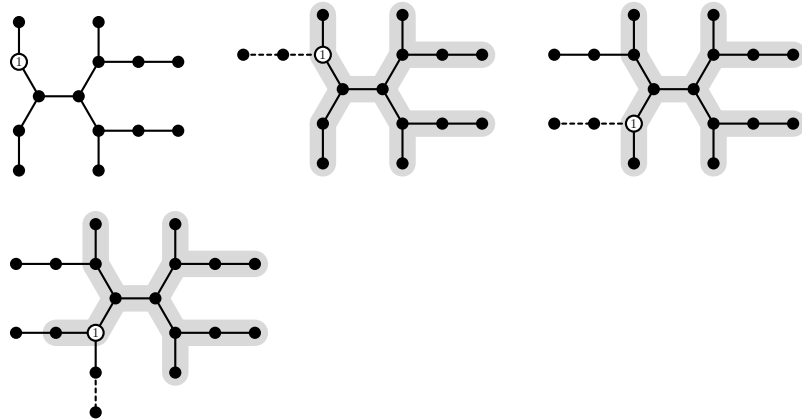

  \centering
    \mpfile{1root}{10}\qquad \mpfile{1root}{11}

    \bigskip
    
    \mpfile{1root}{12}
    \qquad \mpfile{1root}{13}
    
    \caption{A not mortal amoeba with one root with respect to $2$-extensions. The next step uses the left lower $P_5$ for extensions. Repetitions of the process continue forever}
    \label{pic:1root-P3}
  \end{figure}

It is clear that the last tree in any generation sequence is confining. Thus, mortality implies the existence of a confining tree. 
The converse is proven only for $\ell\in \{1,2\}$ and it remains open in general.

\begin{lemma}\label{lm:suff-conf-mort}
For extension parameter $\ell\in \{1,2\}$, an amoeba is mortal if and only if there exists a confining tree for it.
\end{lemma}
\begin{proof}
  The `if part' is obvious. 
  The `only if part': take a confining tree and a copy of the amoeba. Start a generation sequence from this copy. 
  Note that $1$- and $2$-growths of a copy of the amoeba in any tree are isomorphic. 
  
  Indeed, if an extension path from a root contains a new edge then the root should be a pendant vertex in the initial tree or should be adjacent to a pendant vertex since $\ell\leq2$. Thus, we add either a pendant edge or a path of length $2$.  
  In the tree that is the union of two different growths these edges or paths belong to the same orbit of automorphisms. Therefore, the different growths are isomorphic to each other.
This implies that at each step of extension new edges can be chosen from the confining tree. Therefore, the sequence is finite.
\end{proof}

\begin{remark}\label{rem:12}
    In the general case it is open whether the immortality is  the negation of mortality. 
        For the extension    parameter  $\ell\in
    \{1,2\}$ it is true due to \cref{lm:suff-conf-mort}. 
    If an amoeba is not  immortal then the last tree of its 
     finite generation sequence is a confining tree for the
     amoeba. 
     Thus, all its generation sequences are finite. 
     So the amoeba is mortal.
\end{remark}

In the proof of the next theorem we will use the notion of full $t$-ary tree of depth $\ell$.
To this end we denote  by $T^{t,\ell}$  a rooted  
tree in which every node at distance  less than $\ell$ from the root has $t$ children,  while
every node at distance  $\ell$  is a leaf.

\begin{theorem}\label{th:mortality-confinement}
   An amoeba is 
   mortal w.r.t all trees
   if and only if it is totally confined, regardless of its extension parameter. 
\end{theorem}
\begin{proof}
Let $\ell$ be the extension parameter.

  In one direction it is easy. Suppose an amoeba is   mortal w.r.t all trees. 
  Take a tree $T$ and any $\ell$-growth sequence $T=T_0, \dots, T_k$ starting in $T$. It is finite and the last tree $T_k$ in the sequence confines the amoeba and contains $T$.

For the other direction, take a tree $T_0$ and a totally confined amoeba $A = (H, m)$. 
Let $\delta =  \max_{v\in V(H)} \big(d(v)+m(v) \big)$, $d = \diam {T_0} $. 
Choose a confining tree $T$ that contains a full $\delta$-ary tree $T^{\delta, 3\ell+2d}$ of sufficiently large depth, say, $3\ell+2d$.  
Choose a copy of $T_0$ containing the  center of $T^{\delta, 3\ell+2d}$. 
Start $\ell$-growth sequence $T_0, T_1, \dots$ from this particular copy. 
If $\diam {T_i}\leq 2\ell-6$ then any copy of $A$ in it is at distance at most $2\ell-6$ from the center of $T^{\delta, 3\ell+2d}$. 
Thus, any  growth of $T_i$ can be embedded to $T$. 
If an $\ell$-growth of $T_i$ uses edges of $T_0$, then the corresponding copy of $A$ is at distance at most $2d$ from the center of $T^{\delta, 3\ell+2d}$. 
Thus, again, the growth can be embedded to $T$. 
If an $\ell$-growth of $T_i$ does not use edges of $T_0$ and $\diam {T_i}>2\ell-6$, then \cref{lm:depth-bound} implies that all such growths are isomorphic. 
So we apply the same arguments as in the proof of \cref{lm:suff-conf-mort} and conclude that $T_{i+1} $ can be embedded in $T$. 
Therefore, the sequence is finite.
\end{proof}

\subsection{Orbit lemma}\label{subs:orbit}
Fix an amoeba $A=(H,m)$.
For a vertex $v\in V(H)$, we define the \emph{orbit} of $v$, denoted by $\orb v$, as the set of vertices $w\in V(H)$ such that there exists an automorphism of $H$ mapping $v$ to $w$. 
The \emph{completion} $[A]$ of amoeba $A$ is defined as  
\[
[A]=(H,m'), \qquad m'(v) = \max_{w\in \orb{v}} m(w).
\]

We call a pair of amoebas  $A=(H,m)$ and $A'=( H,m')$ \emph{equivalent} (to each other) if $[A]=[A']$.
Using this notation one can easily obtain the following claim.

\begin{lemma}\label{orbit-lemma}
Let $\ell$ be the extension parameter and $A=(H,m)$ be an amoeba.
If $\ell\in \{1,2\}$ then $A$ is mortal if and only if its completion is.
For other values of $\ell$, an amoeba is
 mortal w.r.t all trees
if and only if its completion is.
\end{lemma}
\begin{proof}
  We start from the case $\ell\in \{1,2\}$.   Suppose that  $[A]$ is mortal, that is, it admits a confining tree $T'$.   It remains to observe that $T'$ confines also $A$, so $A$ is mortal by \cref{lm:suff-conf-mort}.

  Now suppose that $A$ is mortal. \cref{lm:suff-conf-mort} implies that there exists a~confining tree $T$ of $A$. Let  $(H,m')$ be a copy of $[A]$ in $T$, $v\in V(H)$, and $v^*$ be a vertex in $\orb {v}$ such that $m(v^*) = m'(v)$. 
  By the choice of $v^*$, there exists an automorphism of $H$ mapping $v^*$ to $v$. The corresponding copy of the amoeba $A=(H,m)$ has an $\ell$-extension inside $T$. It means that $H$ can be extended by  $m'(v)$ paths of length $\ell$ in $T$ such that  the paths have the only common vertex $v$ and each of these paths has the only common vertex $v$ with $H$.

  By using such paths at each vertex of $H$, one can construct an $\ell$-extension of any copy of $[A]$ inside $T$. It means that $T$ confines $[A]$.

For $\ell>2$, similar arguments show that an amoeba is 
not mortal w.r.t some tree
if and only if its completion is. The only difference is in applying \cref{th:mortality-confinement} instead of \cref{lm:suff-conf-mort}.
\end{proof}

\section{Case of $1$-extensions}\label{sec:1-extensions}

Understanding the mortality of general amoebas seems to be very difficult. 
In this section we focus on the mortality conditions for amoebas which grow in the simplest of ways -- by attaching pendant edges.\footnote{In our notation such amoeba growth corresponds to the case of $1$-extensions.}
In \cref{rem:12} we explained that for this case, immortality is just the negation of mortality.
The results of  \cite{gurvich2023avoidability} easily imply that 
any immortal amoeba is subcubic whenever the number of roots is equal to $2$.
The following subsection generalizes those results to the case of arbitrary number of roots.

\subsection{Degree conditions}\label{subs:degree}

  For the rest of this subsection fix  an immortal amoeba  $A=(H,m)$.
  Let $(d(v), m(v))$, $v\in H$, be the sequence of vertex degrees and multiplicities of the roots in $H$. We also define
$\tm=\max_{v\in V(H)}(d(v) +m(v))$.

  Fix an infinite $1$-generation sequence $G_0, G_1, \dots$.
  The next claim immediately follows from the definition of the $\ell$-extension.

  \begin{claim}
    If a vertex $v$ of $G_i$ is a root of a copy of the amoeba to be extended and 
    $d_{G_{i+1}}(v)>d_{G_i}(v)$ then $d_{G_i}(v)< d_{H}(v)+m(v)$.
  \end{claim}
  \begin{proof}
    If $d_{G_i}(v)\geq d_{H}(v)+m(v) $ then an extension uses existing edges incident to $v$ due to the minimality condition.
  \end{proof}

  \begin{cor}\label{corr:degree-bounded}
    $d_{G_i}(v)\leq \tm$ for all $i$ and $v\in V(G_i)$.
  \end{cor}

This corollary provides a uniform upper bound on vertex degrees in all graphs~$G_i$. From this  we derive the following.

  \begin{cor}\label{saturation}
    For all $i$ and sufficiently large $j$, vertices of a copy extending at step $j$ do not intersect $V(G_i)$.
  \end{cor}
\begin{proof}
    If a vertex $v\in V(G_i)$ belongs to a copy of the amoeba  extending at some step, then a new pendant edge is added at distance at most $|V(H)|$ from $v$ in $G_{i+1}$. Since all degrees of vertices in all graphs in the generation sequence are bounded by $\tm$ due to \cref{corr:degree-bounded}, it can repeat 
    finitely many times. 
\end{proof}
These facts put restrictions on degrees of vertices in $H$. 
Recall that we add pendant edges. So, initially a degree of a new vertex $v$ is 1. 
Then it can be increased by extension of a copy having a root at~$v$. 
Thus, $\min_{v: m(v)>0}d(v) = 1$.

To catch the process of degree increasing, we introduce an auxiliary 
directed graph $D_A$. 
The vertices of $D_A$ are positive integers up to $\tm$. 
A pair  $(x,y)$ is an edge in $D_A$ if $y = d(v)+m(v)$ and $d(v)+m(v)> x\geq d(v)$ for some~$v\in V(H)$.

Let $q_A$ be the maximal integer reachable from 1 in $D_A$. 
By the definition of $q$, the inequality $q\leq \tm$ holds. 
For an immortal amoeba, 
let $v^*$ be such that
$\tm = d(v^*)+m(v^*)$.
Note that $\tm$ is reachable because  $d(v^*)$  appears in a new vertex (i.e. a~vertex from $V(G_i)\setminus V(G_0)$) and all copies containing this  vertex at some moment have enough space for the extension. This implies the following theorem.

\begin{theorem}\label{th:degree-condition}
  Let
  $A=(H, m)$ be an immortal amoeba 
  with its auxiliary digraph $D_A$. 
  Then
$q_A=\tm$ and $d_H(v)\leq \tm$ for all $v\in V(H)$.
\end{theorem}

\begin{claim}
  For any amoeba $A$ with $k$ roots, we have $\tm\leq 1+k$.
\end{claim}
\begin{proof}
  Let $x_1, x_2, \dots , x_t$ be a directed path in $D$. 
  By definition of $D$, for all $i\in [1,t]$ there  exists $v_i\in V(H)$ such that 
  $x_i= d(v_{i})+m(v_i)$ and $x_{i-1}\geq d(v_{i})$, for $i>1$. Then
\[
\begin{aligned}
  &x_1\leq 1;\\
  &x_2= d(v_{2}) + m(v_{2}) \leq 1+ m(v_{2});\\
  &\dots\\
  &x_t\leq 1+ \sum_{j=2}^t m(v_j)\leq 1+\sum_{v\in V(H)}m(v) = 1+k.
\end{aligned}
\]
The last inequality follows from the fact that each $x_j$ can appear along the path only once.
\end{proof}

\begin{cor}
Let  $A=(H,m)$  be an immortal amoeba such that $k=1$. Then  $\Delta(H)\leq 2$.
\end{cor}

\begin{cor}
Let $A=(H,m)$ be an amoeba.
If $\Delta(H)>k+1$, then $A$ is mortal.  
\end{cor}

\noindent
The following generalization of the above corollary remains open.

\begin{conjecture}
For any value of the extension parameter $\ell$,
and amoeba $A=(H, m)$, let $k$ be the number of roots in $A$.
If $\Delta(H)>k+2$ then $A$ is totally confined. 
\end{conjecture}

\subsection{Mortality criteria for caterpillars}\label{subs:caterpillar}
In this subsection a $1$-extension of a given amoeba will be referred to simply as an \emph{extension}. 
Furthermore, throughout this section we assume that the multiplicity function $m$ takes values $0$ and $1$ only. 
A \emph{caterpillar} $C(d_1,d_2,\dots,d_\ell)$ is a tree constructed as follows.
We start with a \emph{central path} $P_\ell=(p_1,\dots,p_\ell)$, and afterwards, for each $i\in \{1,\dots,\ell\}$, we add $d_i$ pendant vertices to $p_i$.
In addition, we require $d_1=d_\ell=0$.
A caterpillar is said to be symmetric if $d_i=d_{\ell-i+1}$, for all $i$. 
The family of amoebas $(H,m)$ where $H$ is a caterpillar and $m$ respects the above-mentioned properties is denoted by $\mathcal A$.
In this section we characterize mortality of amoebas from $\mathcal A$.

\begin{observation}
Let $H$ be a caterpillar with central path $P\subseteq V(H)$, and let $A=(H,m)$ be an amoeba such that $\supp(m)\subseteq P$. 
An extension of $A$
gives rise to a caterpillar.
\end{observation} 
\begin{figure}[!h]
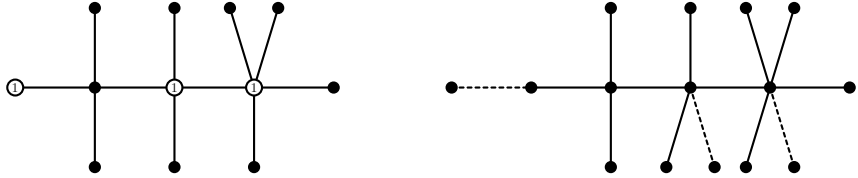

  \centering
    \mpfile{cat}{10}\qquad\qquad \mpfile{cat}{11}

    \caption{A caterpillar $(C(0,2,2,3,0), m)$ and its extension $C(0,0,2,3,4,0)$.}
    \label{pic:cat}
  \end{figure}
\begin{proof}
It is enough to observe that the central path of the obtained graph corresponds to $P$, possibly with an additional vertex on either endpoint of $P$ (also see \cref{pic:cat}).
\end{proof}

We now give a necessary and sufficient condition for an amoeba from $\mathcal A$ to be shiftable `to the right'.
We start with some definitions.

\begin{definition}[of slow integer sequences]
    Let $S=(d_1,d_2,\dots)$ be an integer sequence. 
    We say that $S$ is 
    \begin{itemize}
        \item     \emph{slowly decreasing}, if $d_{i}-d_{i-1}\ge -1$ for all $d_i,d_{i-1}$; 
        \item     \emph{slowly increasing}, if $d_{i}-d_{i-1}\le 1$ for all $d_i,d_{i-1}$; 
        \item \emph{slow} if it is either slowly increasing or slowly decreasing.
    \end{itemize}
\end{definition}
\begin{definition}[of slow amoebas]
    For an integer sequence $S$
    let $(p_1,\dots,p_\ell)$ be the central path  of $C(S)$ and $(C(S),m)$ be an amoeba from  $\mathcal A$.
    The amoeba $(C(S),m)$ is said to be
\begin{itemize}
\item  \emph{slowly decreasing} if $S$ is  slowly decreasing 
and the roots are placed at $p_i$ whenever $d_i-d_{i-1}=-1$, and also in $p_\ell$;
\item \emph{slowly increasing} if $S$ is   slowly decreasing and the roots are placed at $p_{i-1}$ whenever $d_i-d_{i-1}=1$, and also in~$p_1$; 
\item \emph{slow} if it is either slowly increasing or slowly decreasing. 
\end{itemize}
\end{definition}

\begin{theorem}
An amoeba $A\in \mathcal A$ is mortal if and only if $[A]$ is not slow. 
\label{thm:non-symmetr-amoebas}
\end{theorem}

\begin{proof}
    Let $S=(d_1,d_2,\dots)$ be the degree sequence of the central path of $[A]$. 
    To prove the `if direction' assume w.l.o.g. that $[A]$ is slowly decreasing. 
    We define a sequence of graphs $\mathcal G= (G_1,G_2,\dots)$,  
    starting with  $G_1=C(S)$. 
    (Informally speaking, $\mathcal G$ is a sequence of `right shifts' of $[A]$.)
    Every member of this sequence is a caterpillar graph, hence, the vertex notation (vertices $p_1,\dots$) is well defined.  
    
    Within $G_i$ define a~set of vertices $\mathcal C_i\subseteq V(G)$ as $\mathcal C_i=\{p_i,\dots, p_{i+|S|}\}$.

    If $G_i$ admits a copy $H$ of $C(S)$ containing  $\mathcal C_i$ then  $G_{i+1}$ is obtained from $G_i$ by extending $H$. 

   Due to \cref{lm:suff-conf-mort} of the `if direction' of the statement, it is enough to prove that $\mathcal G$ is infinite. 
   Assume for contradiction that $\mathcal G$ has (a finite) length $k$.
    By definition $G_1=C(S)$ 
    contains $\mathcal C_1$, hence, $k\ge 2$.

    Furthermore, $G_k$ does not admit a copy of $C(S)$ containing  $\mathcal C_k$. In other words, either  
    \begin{enumerate}[(i)]
           \item there is an~$i$ such that  $d_{G_k}(p_{k+i})<d_i+2$ with $i<|S|$, or\label{it:one}
        \item  $d_{G_k}(p_{k+i})<d_i+1$ with $i=|S|$.\label{it:two}
 \end{enumerate}
  Note that (\ref{it:two})   implies $d_{G_k}(p_{k+|S|})=0$, which is impossible.
Let us fix an~$i$ satisfying    (\ref{it:one}).

    By construction, there exists a copy $H$ of $C(S)$ in $G_{k-1}$ containing $\mathcal C_{k-1}$ whose extension gives rise to $G_k$. By~(\ref{it:one}) we have $d_i-d_{i-1}<0$;  furthermore, $d_i-d_{i-1}=-1$, since $S$ is slowly decreasing.
    However, due to the extension of $[A]= (H,m')$ in $G_{k-1}$, we are guaranteed that 
    $m'(p_{k+i})>0$,
    and hence
    $d_{G_k}(p_{k+i})-2\ge d_i$, which is a contradiction.

To prove the `only if direction' of the claim assume that $[A]=(H,m')$ is neither slowly increasing nor slowly decreasing.
This can be due to one of the following reasons. There exist:
\begin{enumerate}
    \item  $i$ and $j$ such that 
    $d_i-d_{i-1} <-1$, while 
    $d_j-d_{j-1} > 1$;
    \item  $i$ and $j$ such that 
    $d_i-d_{i-1} <-1$, and 
    $d_j-d_{j-1} = 1$, while $p_{j-1}$ is not a root;
    \item   $i$ and $j$ such that 
    $d_i-d_{i-1} =-1$, and 
    $d_j-d_{j-1} > 1$, while
    $p_i$ is not a root;
    \item  $i$ and $j$ such that 
    $d_i-d_{i-1} =-1$ and 
    $d_j-d_{j-1} = 1$, while neither $p_i$ nor $p_{j-1}$ are roots.
\end{enumerate}

Consider the central path $(p_1,\dots,p_{\ell})$ in an arbitrary copy of $[A]$ in $G_1$.
In all four above cases
the following claim holds:
\begin{lemma}
    For all copies of $[A]$ in $G_2$, the sets $\{p_2,\dots,p_{\ell-1}\}$ coincide. 
    \label{lem:short}
\end{lemma}
\begin{proof}
    Indeed, any path of length $\ell-1$ in $G_2$ contains   a pendant vertex either  from $V(G_2)\setminus V(G_1)$ or  from $G_1$, and in the latter case it is   adjacent to either $p_2$ or $p_{\ell-1}$.
In fact, the former case of this alternative is possible only if the amoeba is slow. Otherwise, one of the cases listed above holds. 
\end{proof}

Denote by $(p'_1,\dots,p'_{\ell})$ the vertices of $G_2$ which correspond to $(p_1,\dots,p_{\ell})$ of $A$. 
By \cref{lem:short}, 
for $1< i\leq \ell$, we have that either
\[
p'_i = p_{i-1}, \text{ or}\quad
p'_i = p_{i}, \text{ or}\quad
p'_i = p_{\ell-i+1}, \text{ or}\quad
p'_i = p_{\ell-i}.
\]
In all four cases, we obtain a contradiction with the vertex degrees' conditions in $G_2$.

To finish the proof, take a copy $A'$ of $[A]$ in $G_2$ such that its central path contains a pendant vertex $v$ from $G_1$  adjacent to $p_{\ell-1}$; the case when it is adjacent to $p_2$ is proved similarly.
To have enough room for $A'$ to grow, we add one pendant edge to $v$ only. 
Notice that no new copies of $[A]$ appear, thus $[A]$ is mortal, and due to \cref{orbit-lemma}, so is $A$. 
\end{proof}
\newpage
\section{Directions for future  work}\label{sec:future}
As announced in \cref{sec:amoebas}, we conjecture the following.
\begin{conjecture}
    
    The following  properties of an amoeba $A$ are equivalent:
    \begin{itemize}
        \item $A$ is mortal.
        \item $A$ is not immortal.
        \item $A$ is  mortal w.r.t. any tree.
    \end{itemize}
\end{conjecture}

Deciding mortality looks hard (undecidable); perhaps one can model a~general Turing machine by amoebas. For this purpose a more general notion of an amoeba colony might be useful. 
An amoeba colony is a set of amoebas $\cA = \{A_1,\dots, A_n\}$. 
An \emph{$\ell$-growth} $G'$ of $\cA$ in $G$ is an $\ell$-growth of  some $A_i\in\cA$. 
The remaining definitions are the same as for the case of  amoebas. So, in construction of an $\ell$-growth sequence for an amoeba colony, one can choose an alive amoeba to be extended. It gives more freedom and we conjecture that such a~process can simulate a universal computation.

\begin{conjecture}
  The following algorithmic problem is undecidable.

\begin{description}
    \item\textsc{Input:} an integer $\ell$, a tree $T$ and an amoeba colony $\cA$.

\item\textsc{Question:} is there an infinite growth sequence for $\cA$ starting from $T$?
\end{description}
\end{conjecture}

The following generalizations  might be of independent interest.

A more general type of an extension of an amoeba is to attach to each vertex a specified tree. A general-type amoeba is $(T, \pi)$, where $\pi $ is a \emph{root function} from $V(T)$ to rooted trees (the empty tree is allowed). All other definitions are naturally extended to the case of general-type amoebas.

A general-type amoeba is called  \emph{absolutely mortal} if  it is mortal for any root function $\pi$. Characterization of absolutely mortal amoebas is an interesting open question in this area.

Another way of generalization is to substitute trees by arbitrary graphs, possibly disconnected. We know a little about this general problem. But it also seems to be interesting.

As already mentioned in \cref{rem:confining-graphs}, we can extend the definition of confining trees to general graphs. This would naturally generalize the concept of totally confined amoebas. 
From our previous results \cite{gurvich2023avoidability} it follows that there exist amoebas which do not admit confining trees but can be confined by a general graph. For various amoebas, the confining graphs are cages: Petersen graph, Heawood graph, McGee graph, Tutte–Coxeter graph.
The confinement of general amoebas might lead to interesting results about cages and other graphs having rich structure. 

\paragraph{Acknowledgements.}
The authors are grateful to Martin Milanič for valuable discussions and observations, Susan D. Cook for English proofreading, and to an anonymous reviewer for many helpful remarks. 

This research was partially prepared within the framework of the HSE University Basic Research Program. 
This work is supported in part by the Slovenian Research Agency (research program P1-0383 and research projects N1-0160, N1-0209, J1-3003, J1-4008 and J5-4596).
The work of the third author was partially supported by State Assignment, theme no. FFNG-2024-0003.

\end{document}